\documentclass{amsart}

\usepackage{amsmath,amssymb,amsthm}
\usepackage{a4wide}

\def\ocirc#1{\ifmmode\setbox0=\hbox{$#1$}\dimen0=\ht0
\advance\dimen0 by1pt\rlap{\hbox to\wd0{\hss\raise\dimen0
\hbox{\hskip.2em$\scriptscriptstyle\circ$}\hss}}#1\else
{\accent"17 #1}\fi}

\newtheorem{theorem}{Theorem}

\newtheorem{lemma}[theorem]{Lemma}

\newtheorem{definition}[theorem]{Definition}

\newcommand{\R}{\mathbb{R}}

\newcommand{\T}{\mathbb{T}}

\begin{document}

\title[Existence of solutions for dynamic inclusions on time scales]{%
Existence of solutions for dynamic inclusions on time scales via duality}

\thanks{Submitted 29-Aug-2011; revised 09-Jan-2012;
accepted 21-Jan-2012; for publication in \emph{Applied Mathematics Letters}.}

\author[E. Girejko]{Ewa Girejko}

\address{Ewa Girejko\newline
\indent Faculty of Computer Science, Bia{\l}ystok University of Technology,
15-351 Bia\l ystok, Poland}

\email{e.girejko@pb.edu.pl}


\author[D.F.M. Torres]{Delfim F. M. Torres}

\address{Delfim F. M. Torres\newline
\indent Center for Research and Development in Mathematics and Applications\newline
\indent Department of Mathematics, University of Aveiro, 3810-193 Aveiro, Portugal}

\email{delfim@ua.pt}


\begin{abstract}
The problem of existence of solutions to nabla differential equations
and nabla differential inclusions on time scales is considered.
Under a special form of the set-valued constraint map,
sufficient conditions for the existence of at least one solution,
that stay in the constraint set, are derived.
\end{abstract}




\maketitle


\section{Introduction}

Recently, motivated by applications in population dynamics and economics,
the theories of differential and difference equations
have been unified and extended into the study of dynamic equations
on time scales \cite{A:B:O:P:02,Atici,A:H,MalTor:comp}.
The new theory illustrates and explains the discrepancies
between continuous and discrete-time results, and provides more general
results on an arbitrary closed subset of the real numbers \cite{MR2410768,MR2445270,MR2604248}.
In this way, results apply not only to the set of reals or set of integers,
but to more general time scales such as the Cantor or quantum sets
\cite{B:P:01,B:P:03,R:D,Lak:book,MR2733985}.

In the theory of differential equations,
a useful technique to prove existence of solutions
consists in proving necessary and/or sufficient
conditions under which at least one trajectory of a vector field,
starting at a point of a certain constraint set, stays in that set \cite{Aubin}.
Existence results of this kind for dynamic equations on time scales
are a rarity. We are only aware of \cite{D:09},
where a generalization of Wa\.zewski's
method for systems of delta dynamic equations
on time scales is obtained. The main result
of \cite{D:09} is reviewed in Section~\ref{sec:prelim}
(Theorem~\ref{th:Diblik}). For related studies on dynamic inclusions
on time scales, the reader is referred to \cite{MR2797040,MR2050226,MR2217094}.

Here we formulate sufficient conditions for nabla dynamic equations
on an arbitrary time scale, guaranteeing the existence
of at least one trajectory of a vector field starting
at a point of a set $K$ (constraint set) and staying in that set
(Theorem~\ref{th:Diblik:nabla}). Moreover, using a Filippov theorem
on time scales \cite{P:T:10} and its dual version in the sense
of \cite{C:10} (see also \cite{MalTor,MR2794990,PawTor}),
we extend the results to special cases of nabla
(Theorem~\ref{th:viab:nabla}) or delta differential inclusions
(Theorem~\ref{th:viab:delta}). We end with an example.


\section{Preliminaries}
\label{sec:prelim}

For an introduction to the theory of time scales we refer the reader to
\cite{A:B:O:P:02,B:P:01,B:P:03,SH,Lak:book}. Here we just recall
some necessary and recent results of \cite{C:10,D:09,P:T:10}.
Throughout the paper we assume $\T$ to be a given time scale with
jump functions $\sigma$ and $\rho$, forward graininess $\mu$
and backward graininess $\nu$, differential operators
$\Delta$ and $\nabla$, and where $\inf\T:=a$, $\sup\T:=b$,
and $[a,b]_{\T}:=[a,b]\cap\T$.


\subsection{Caputo's duality approach}

The notion of dual time scale was introduced in \cite{C:10}
and has shown to be a very useful concept in control theory \cite{PawTor}
and the calculus of variations \cite{MalTor,DNA}.

\begin{definition}[Dual time scale]
Given a time scale $\T$, we define the \emph{dual time scale} $\T^{\star}$ by
$\T^{\star}:= \{ s\in \R | -s\in \T\}$.
\end{definition}

\begin{lemma}[\cite{C:10}]
If $a,b\in\T$ with $a<b$,
then $\left([a,b]_{\T}\right)^{\star}=[-b,-a]_{\T^{\star}}$.
\end{lemma}

\begin{definition}[Dual function]
Given a function $f:\T\rightarrow \R$ defined on the time scale $\T$,
the \emph{dual function} $f^{\star}:\T^{\star}\rightarrow\R$ on the dual time scale
$\T^{\star}$ is defined by $f^{\star}(s):=f(-s)$ for all $s\in\T^{\star}$.
\end{definition}

\begin{lemma}[\cite{C:10}]
\label{lem1}
If $ \sigma$, $\rho:\T\rightarrow \T$ are the jump operators for $\T$,
then the jump operators for $\T^{\star}$,
$\hat\sigma$, $\hat\rho:\T^{\star}\rightarrow\T^{\star}$,
are given by $\hat{\sigma}(s)=-\rho(-s)$ and
$\hat{\rho}(s)=-\sigma(-s)$ for all $s\in\T^{\star}$.
\end{lemma}

Let $\mathbb{T}$ be a time scale.
If $\sup \mathbb{T}$ is finite and left-scattered, one defines
$\mathbb{T}^\kappa := \mathbb{T}\setminus \{\sup\mathbb{T}\}$,
otherwise $\mathbb{T}^\kappa :=\mathbb{T}$. Similarly,
a new set $\mathbb{T}_\kappa$ is derived from $\mathbb{T}$ as
follows: if  $\mathbb{T}$ has a right-scattered minimum $m$, then
$\mathbb{T}_\kappa=\mathbb{T}\setminus\{m\}$; otherwise,
$\mathbb{T}_\kappa= \mathbb{T}$.

\begin{lemma}[\cite{C:10}]
Given a time scale $\T$, then
$(\T^{\kappa})^{\star}=(\T^{\star})_{\kappa}$, and
$(\T_{\kappa})^{\star}=(\T^{\star})^{\kappa}$.
\end{lemma}

\begin{lemma}[\cite{C:10}]
\label{lem2}
Given $ \mu:\T\rightarrow \R$, the forward graininess of $\T$,
then the backward graininess of $\T^{\star}$, $\hat \nu:\T^{\star}\rightarrow \R$,
is given by the identity $\hat\nu (s)=\mu^{\star}(s)$
for all $s\in\T^{\star}$. Similarly, given $\nu:\T\rightarrow \R$,
the backward graininess of $\T$, then the forward graininess of $\T^{\star}$,
$\hat \mu:\T^{\star}\rightarrow \R$, is given by the identity
$\hat\mu (s)=\nu^{\star}(s)$ for all $s\in\T^{\star}$.
\end{lemma}

\begin{lemma}[\cite{C:10}]
\label{lemmacont}
Function $f:\T\rightarrow \R$ is rd-continuous (resp. ld-continuous) if and only if its dual
$f^{\star}:\T^{\star}\rightarrow \R$ is ld-continuous (resp. rd-continuous).
\end{lemma}

\begin{lemma}[\cite{C:10}]
\label{lemmaderi}
If $f:\T\rightarrow \R$ is delta (resp. nabla) differentiable at $t_0\in\T^{\kappa}$
(resp. at $t_0\in\T_{\kappa}$), then $f^{\star}:\T^{\star}\rightarrow \R$ is nabla
(resp. delta) differentiable at $-t_0\in(\T^{\star})_{\kappa}$
(resp. at $-t_0\in(\T^{\star})^{\kappa}$), and the following equalities hold true:
$f^{\Delta}(t_0)=-(f^{\star})^{\hat\nabla}(-t_0)$
(resp.  $f^{\nabla}(t_0)=-(f^{\star})^{\hat\Delta}(-t_0)$),
$f^{\Delta}(t_0)=-((f^{\star})^{\hat\nabla})^{\star}(t_0)$
(resp. $f^{\nabla}(t_0)=-((f^{\star})^{\hat\Delta})^{\star}(t_0)$),
and $(f^{\Delta})^{\star}(-t_0)=-((f^{\star})^{\hat\nabla})(-t_0)$
(resp. $( f^{\nabla})^{\star}(-t_0)=-(f^{\star})^{\hat\Delta}(-t_0)$),
where $\Delta$ and $\nabla$ denote the delta and nabla derivative for the time scale $\T$ while
$\hat\Delta$ and $\hat\nabla$ denote the delta and nabla derivative for the time scale $\T^{\star}$.
\end{lemma}

\begin{lemma}[\cite{C:10}]
\label{lemmacontder}
Let $f:\T\rightarrow \R$. Function $f$ belongs to $C^1_{rd}$ (resp. $C^1_{ld}$)
if and only if its dual $f^{\star}:\T^{\star}\rightarrow \R$
belongs to $C^1_{ld}$ (resp. $C^1_{rd}$).
\end{lemma}


\subsection{Set-valued maps and dynamic inclusions}
\label{subsec:aux}

Let $K:\T\twoheadrightarrow\R^n$ be a set-valued map such that
$K(t)=\{x\in\R^n : b_i(t)< x_i< c_i(t),\, i=1, \ldots, n\}$,
where $b_i,c_i:\mathbb{T} \rightarrow \mathbb{R}$, $i=1, \ldots, n$,
are delta differentiable functions on $\T$ such that $b_i(t) < c_i(t)$
for each $t\in \T$. The graph of $K$, $Graph(K)\subseteq\T\times\R^n$,
is defined by $Graph(K):=\{(t,x):t\in\T,\ x\in K(t)\}$.

Let us consider the control system
\begin{equation}
\label{eq:00}
x^\Delta(t)=f(t,x(t),u(t)),
\quad u\in\mathcal{U},
\end{equation}
where $\mathcal{U}=\{u(\cdot) \ | \ u(\cdot) \text{ is }
\Delta\text{-measurable\;piecewise\;rd-continuous\;with }
u(t)\in U\subset\R^m\;
\text{for\;all}\;t\in\T\}$
is the set of admissible controls. Suppose that the set
$U\subset\R^m$ of control values
is compact and function $f:Graph(K) \times U \rightarrow \R^n$
is rd-continuous with respect to the first variable and continuously
differentiable with respect to the second variable $x\in\R^n$.
For a fixed $u$ we can rewrite \eqref{eq:00} and consider the dynamic system
\begin{equation}
\label{eq:0}
x^\Delta(t) = f_u(t,x(t)),
\end{equation}
where $f_u=(f_1, \ldots, f_n):\T\times\R^n\rightarrow \R^n$.
Let us assume that $f_u$ is continuous. Moreover, for every fixed
non-isolated point $t\in \T$, let $S_t\subseteq\T\times\R^n$ be a closed set,
$[t-a, t+a]\times \overline{K(t)} \subset S_t$ for an $a > 0$, $\inf \T\leq t-a$,
$\sup \T \geq t + a$, such that $f_u$ is rd-continuous, bounded and Lipschitz
continuous on $S_t$.
Let $\partial Graph(K):=\{(t,x) : t\in\T,\, x\in \partial K(t)\}$ with
$\partial K(t)=\overline{K(t)}\backslash K(t)$, where
$\overline{K(t)}=\{x\in\R^n:b_i(t)\leq x_i\leq c_i(t),\, i=1, \ldots, n\}$.
By a \emph{trajectory} of the system \eqref{eq:00}
from $x_0$ corresponding to the control $u\in \mathcal{U}$ we mean the  function
$x=\psi(t_0,\cdot,x_0,u) :[t_0^u,t_1^u]_{\T} \rightarrow \R^n$
such that $x$ is the unique solution of the initial value problem
$x^{\Delta}(t)=f_u(t,x(t))$, $x(t_0)=x_0$, provided it is defined
for all $t \in [t_0,t_1]_{\T}$ and $x(t)\in K(t)$.
We define the auxiliary functions
$B_i(t,x):=-x_i+b_i(t)$ and $C_i(t,x):=x_i-c_i(t)$ on $\T\times\R^n$,
$i\in \{1, \ldots, n\}$, and sets
$Graph(K)^i_B:=\{(t,x)\in Graph(K):B_i(t,x)=0\}$
and $Graph(K)^i_C:=\{(t,x)\in Graph(K):C_i(t,x)=0\}$,
where $i\in \{1, \ldots, n\}$. It is clear that
$\partial Graph(K)=\bigcup^n_{i=1}
\left(Graph(K)^i_B\cup Graph(K)^i_C\right)$.

\begin{definition}[Delta-points of strict egress]
A point
$M=(t, x_1, \ldots, x_{i-1}, b_i(t), x_{i+1}, \ldots, x_n)
\in Graph(K)^i_B$, $i\in \{1, \ldots, n\}$,
is called \emph{a delta-point of strict egress for the set} $Graph(K)$
\emph{with respect to system} \eqref{eq:0} if
$f_i(M) < b^\Delta_i(t)$. A point
$M=(t, x_1, \ldots, x_{i-1}, c_i(t), x_{i+1}, \ldots, x_n)
\in Graph(K)^i_C$, $i\in \{1, \ldots, n\}$,
is called \emph{a delta-point of strict egress for the set} $Graph(K)$
\emph{with respect to system} \eqref{eq:0} if $f_i(M) > c^\Delta_i(t)$.
\end{definition}

\begin{theorem}[\cite{D:09}]
\label{th:Diblik}
Let $f_u: \T\times\R^n \rightarrow \R^n$ and $K:\T\twoheadrightarrow \R^n$, where
$K(t)=\{x\in\R^n:b_i(t)< x_i< c_i(t),\, i=1, \ldots, n\}$ with $b_i,c_i:\mathbb{T}
\rightarrow \mathbb{R}$, $i=1, \ldots, n$, delta differentiable functions
on $\T$ such that $b_i(t) < c_i(t)$ for each $t\in \T$. If every point
$M\in \partial Graph(K)$ is a delta-point of strict egress for the set
$Graph(K)$ with respect to system \eqref{eq:0}, then there exists a value
$\overline{x}\in K(t_0)$ such that the initial value problem
$x^\Delta(t)= f_u(t,x(t))$, $x(t_0)=\overline{x}$,
has a solution $x=\overline{x}(t)$ to \eqref{eq:0}
satisfying $(t,\overline{x}(t))\in Graph(K)$ for every $t\in\T$.
\end{theorem}

In order to formulate the delta Filippov theorem we rewrite \eqref{eq:00}
as a delta differentiable inclusion. For that we introduce
the set-valued map $F(t,x)$ defined on the graph of $K$:
$F(t,x):=\{f(t,x,u):u\in U\subset\R^m\}$.
Then equation \eqref{eq:00} takes on the form
\begin{equation}
\label{eq:2}
x^\Delta(t) \in F\left(t,x(t)\right).
\end{equation}

\begin{theorem}[Filippov's theorem \cite{P:T:10}]
\label{th:Del:Ewa}
An absolutely rd-continuous function $x : [t_0,t_1]_{\T}\rightarrow\R^n$ is a
trajectory of \eqref{eq:00} if and only if satisfies
\eqref{eq:2} almost everywhere.
\end{theorem}


\section{Main results}
\label{sec:mr}

We adopt here the backward perspective of nabla calculus,
which is sometimes preferable to the delta (forward) approach
\cite{Ricardo,Atici,Natalka,PawTor}.
Let $P:\T\twoheadrightarrow\R^n$ be a set-valued map such that
$P(t)=\{y\in\R^n:\beta_i(t)> y_i>\gamma_i(t),\, i=1, \ldots, n\}$,
where $\beta_i,\gamma_i:\mathbb{T} \rightarrow \mathbb{R}$, $i=1, \ldots, n$,
are nabla differentiable functions on $\T$ such that
$\beta_i(t) > \gamma_i(t)$ for each $t\in \T$.
Consider the nabla differential control system
\begin{equation}
\label{eq:nabla:1}
y^{\nabla}(t)=g(t,y(t),v(t)),
\quad y(t_0)=y_0,
\end{equation}
where the control set $V\subset\R^m$ is compact,
the controls are measurable functions
$v$ with $v(t)\in V$ almost everywhere in $[t_0,t_1]_{\T}$,
and $g:Graph(P)\times V\longrightarrow \R^n$
is $ld$-continuous with respect to $t$ and continuously
differentiable with respect to $y$. For a fixed $v$
we can rewrite $y^{\nabla}(t)=g(t,y(t),v(t))$
as the dynamic system
\begin{equation}
\label{eq:nabla:0}
y^\nabla(t)= g_v(t,y(t)),
\end{equation}
where $g_v=(g_1, \ldots, g_n):\T\times\R^n\rightarrow \R^n$.
Throughout the paper we assume that $g_v$ is continuous.
For every fixed non-isolated point $t\in \T$,
let $Z_t\subseteq\T\times\R^n$ be a closed set,
$[t-a, t+a]\times \overline{P(t)} \subset Z_t$ for an $a > 0$, $\inf \T\leq t-a$,
$\sup \T \geq t + a$, such that $g_v$ is ld-continuous, bounded and Lipschitz
continuous on $Z_t$. Let $\partial Graph(P):=\{(t,y):t\in\T,y\in \partial P(t)\}$
with $\partial P(t)=\overline{P(t)}\backslash P(t)$,
where $\overline{P(t)}=\{y\in\R^n : \beta_i(t)\geq y_i\geq\gamma_i(t),\, i=1, \ldots, n\}$.
If we introduce the set-valued map
$G(t,y)=\{g(t,y,v):v\in V\subset\R^m\}$
defined on the graph of $P$,
then we can write \eqref{eq:nabla:1}
in the nabla differential inclusion form:
\begin{equation}
\label{eq:nabla:3}
y^{\nabla}(t)\in G(t,y(t)),
\quad y(t_0)=y_0.
\end{equation}
We pose the following questions: (i) Does there exist a control $v$ such that system
\eqref{eq:nabla:0} has a solution $y$ such that $y(t)\in P(t)$ for every $t\in [t_0,t_1]_{\T}$?
(ii) Is such a trajectory a solution of system \eqref{eq:nabla:3}?
In order to answer these questions we begin by obtaining nabla versions of
Theorems~\ref{th:Diblik} and \ref{th:Del:Ewa}.

Let us define the notion of points of strict egress in $\nabla$-calculus.
For that purpose we need the auxiliary functions
$V_i(t,y):=-y_i+\beta_i(t)$ and
$W_i(t,y):=y_i-\gamma_i(t)$ on $\T\times\R^n$,
$i\in \{1, \ldots, n\}$, and sets
$Graph(P)^i_V:=\{(t,y)\in Graph(P):V_i(t,y)=0\}$
and $Graph(P)^i_W:=\{(t,y)\in Graph(P):W_i(t,y)=0\}$,
$i\in \{1, \ldots, n\}$. It is clear that
$\partial Graph(P)=\bigcup^n_{i=1} \left(Graph(P)^i_V\cup Graph(P)^i_W\right)$.

\begin{definition}[Nabla-points of strict egress]
\label{def:nabla}
A point
$N=(t, y_1, \ldots, y_{i-1}, \gamma_i(t), y_{i+1}, \ldots, y_n) \in Graph(P)^i_W$,
$i\in \{1, \ldots, n\}$,
is called \emph{a nabla-point of strict egress for the set} $Graph(P)$
\emph{with respect to system} \eqref{eq:nabla:0} if
\begin{equation}
\label{ineq:1}
g_i(N) < \gamma^\nabla_i(t).
\end{equation}
A point
$N=(t, y_1, \ldots, y_{i-1}, \beta_i(t), y_{i+1}, \ldots, y_n)
\in Graph(P)^i_V$, $i\in \{1, \ldots, n\}$,
is called \emph{a nabla-point of strict egress for the set} $Graph(P)$
\emph{with respect to system} \eqref{eq:nabla:0} if
\begin{equation}
\label{ineq:2}
g_i(N) > \beta^\nabla_i(t).
\end{equation}
\end{definition}

\begin{theorem}[Existence of solutions for the nabla differential equation \eqref{eq:nabla:0}]
\label{th:Diblik:nabla}
Let $g_v: \T\times\R^n \rightarrow \R^n$ and $P:\T\twoheadrightarrow\R^n$,
$P(t)=\{y\in\R^n:\beta_i(t)> y_i>\gamma_i(t),\, i=1, \ldots, n\}$ with
$\beta_i,\gamma_i:\mathbb{T} \rightarrow \mathbb{R}$, $i=1, \ldots, n$,
nabla differentiable functions on $\T$ such that $\beta_i(t) > \gamma_i(t)$
for each $t\in \T$. If every point $N\in \partial Graph(P)$ is a nabla-point
of strict egress for the set $Graph(P)$ with respect to system \eqref{eq:nabla:0},
then there exists a value $\overline{y}\in P(t_0)$ such that the initial value problem
$y^\nabla(t)= g_v(t,y(t))$, $y(t_0)=\overline{y}$,
has a solution $y=\overline{y}(t)$
satisfying $(t,\overline{y}(t))\in Graph(P)$ for every $t\in\T$.
\end{theorem}

\begin{proof}
Let function $g_v:\T\times\R^n \rightarrow \R^n$
be the right-hand side of equation \eqref{eq:nabla:0}.
Then it is a continuous function and, moreover,
it is ld-continuous, bounded, and Lipschitz on a closed set
$Z_t\subseteq\T\times\R^n$, where $[t-a,t+a]\times P(t)\subset Z_t$
for $a>0$ and $P(t)=\{y\in\R^n:\beta_i(t)> y_i>\gamma_i(t)$,
$i=1, \ldots, n\}$ with $\beta_i,\gamma_i:\mathbb{T}
\rightarrow \mathbb{R}$, $i=1, \ldots, n$,
nabla differentiable functions on $\T$ such that $\beta_i(t) > \gamma_i(t)$
for each $t\in \T$ and $i=1, \ldots, n$. By duality, there exists its dual
function $g_v^{\star}:\T^{\star}\times\R^n\rightarrow\R^n$ that is continuous.
Moreover, $g_v^{\star}$ is rd-continuous, bounded, and Lipschitz on a closed set
$Z^{\star}_t\subseteq\T^{\star}\times\R^n$ that is dual to $Z_t$ and such that
$[t-a,t+a]^{\star}\times P^{\star}(t)\subset Z^{\star}_t$ for $a>0$
and $P^{\star}(t):=\{x\in\R^n:\beta_i^{\star}(t)< x_i
<\gamma_i^{\star}(t),\, i=1, \ldots, n\}$ with
$\beta_i^{\star},\gamma_i^{\star}:\mathbb{T^{\star}}
\rightarrow \mathbb{R}$, $i=1, \ldots, n$,
delta differentiable functions on $\T^{\star}$ such that
$\beta_i^{\star}(t) < \gamma_i^{\star}(t)$
for each $t\in \T^{\star}$ and $i=1, \ldots, n$.
By assumption, every point $N\in \partial Graph(P)$ is the nabla-point
of strict egress for the set $Graph(P)$ with respect
to system \eqref{eq:nabla:0}. By duality,
one can show that every point $N^{\star}\in \partial Graph(P^\star)$
is a delta-point of strict egress for the set $Graph(P^{\star})$
with respect to system \eqref{eq:0}. Indeed, we have
$(y_i(t)-\beta_i(t))^{\nabla}>0$ and
$-(y^{\star}_i)^{\hat{\Delta}}(-t)+(\beta_i^{\star})^{\hat{\Delta}}(-t)>0$.
Putting $x_i=y^{\star}_i$,  $b_i=\beta_i^{\star}$, $s=-t$,
where $s\in\T^{\star}$ and $t\in\T$, we get
$-x^{\hat{\Delta}}_i(s)+b^{\hat{\Delta}}_i(s)>0$,
$-f_i(N^{\star})+b^{\hat{\Delta}}_i(s)>0$
and $f_i(N^{\star})<b^{\hat{\Delta}}_i(s)$,
for $i=1, \ldots, n$. Analogously, it can be shown for functions
$\gamma_i$ and their dual functions $c_i$, $i=1, \ldots, n$.
Since all assumptions of Theorem~\ref{th:Diblik} are fulfilled, there
exists a value $\overline{x}\in P^{\star}(t_0)$ such that the initial value problem
$x^\Delta(t) = f_u(t,x(t))$, $x(t_0)=\overline{x}$, has a solution $x=\overline{x}(t)$
satisfying $(t,\overline{x}(t))\in Graph(P^{\star})$ for every $t\in\T^{\star}$.
Now, let us take the dual object  $\overline{x}^{\star}$ that is equal to
$\overline{x}^{\star}={\overline{y}^{\star}}^{\star}=\overline{y}$ and we have
$\overline{y}\in P(t_0)$ for $t_0\in\T$. Hence,
the initial value problem $y^\nabla(t)= g_v(t,y(t))$, $y(t_0)=\overline{y}$,
has a solution $y=\overline{y}(t)$ satisfying
$(t,\overline{y}(t))\in Graph(P)$ for every $t\in\T$.
\end{proof}

To prove existence of solutions to the nabla differential inclusion \eqref{eq:nabla:3},
we need the dual analogous of Theorem~\ref{th:Del:Ewa}.

\begin{theorem}[The nabla Filippov theorem on time scales]
\label{th:Del:Ewa:nabla}
An absolutely ld-continuous function $y : [t_0,t_1]_{\T}\rightarrow\R^n$ is a
trajectory of \eqref{eq:nabla:1} if and only if
it satisfies \eqref{eq:nabla:3} almost everywhere.
\end{theorem}

\begin{proof}
If there exists a solution to \eqref{eq:nabla:1}, then it is obvious that it is also
a solution to \eqref{eq:nabla:3}. Let us show, using duality, that
the converse is also true. Assume that there exists an absolutely ld-continuous
function $y : [t_0,t_1]_{\T}\rightarrow\R^n$ solution to \eqref{eq:nabla:3}.
Then there exists also its dual function $x = y^{\star}$, namely
$x : [-t_1,-t_0]_{\T^\star}\rightarrow\R^n$, that is rd-continuous and solution
to equation $x^{\hat{\Delta}}(s)=f(s,x(s),u(s))$, $s\in[-t_1,-t_0]_{\T^{\star}}$,
satisfying $x(s_0)=x_0$ for $s_0\in[-t_1,-t_0]_{\T^\star}$. By Theorem~\ref{th:Del:Ewa},
we know that such a function is a solution to \eqref{eq:2} if and only if it is a solution
to \eqref{eq:00}. Again, by duality, we know that there exists a dual function
to $x$, namely $x^{\star}=y$, $y:[t_0,t_1]_{\T}\rightarrow\R^n$,
such that it is a solution to equation $y^{\nabla}(t)=g(t,y(t),v(t))$, $t\in[t_0,t_1]_{\T}$,
satisfying initial condition $y(t_0)=y_0$ for $t_0\in\T$.
\end{proof}

\begin{theorem}[Existence of solutions to the nabla differential inclusion \eqref{eq:nabla:3}]
\label{th:viab:nabla}
Let $P:\T\twoheadrightarrow\R^n$, $P(t)=\{y\in\R^n : \beta_i(t)>y_i>\gamma_i(t),\, i=1, \ldots, n\}$
with $\beta_i,\gamma_i:\mathbb{T} \rightarrow \mathbb{R}$, $i=1, \ldots, n$,
nabla differentiable functions on $\T$ such that $\beta_i(t) > \gamma_i(t)$
for each $t\in \T$. If every point $N\in \partial Graph(P)$ is a nabla-point
of strict egress for the set $Graph(P)$ with respect to system \eqref{eq:nabla:0},
then there exists a value $\overline{y}\in P(t_0)$ such that the initial value problem
$y^{\nabla}(t)\in G(t,y(t))$, $y(t_0)=\overline{y}$,
has a solution $y=\overline{y}(t)$ satisfying
$(t,\overline{y}(t))\in Graph(P)$ for every $t\in[t_0,t_1]_\T$.
\end{theorem}

\begin{proof}
Since, by assumptions, $P(t)=\{y\in\R^n:\beta_i(t)>y_i>\gamma_i(t),\, i=1, \ldots, n\}$
with $\beta_i,\gamma_i:\mathbb{T} \rightarrow \mathbb{R}$, $i=1, \ldots, n$,
nabla differentiable functions on $\T$ such that $\beta_i(t) > \gamma_i(t)$
for each $t\in \T$ and every point $N\in \partial Graph(P)$ is a nabla point
of strict egress for the set $Graph(P)$ with respect to system \eqref{eq:nabla:0},
one has guarantee of existence of at least one solution to \eqref{eq:nabla:0}
defined on $[t_0,t_1]_{\T}$ such that $(t,y(t))\in Graph(P)$
(or, equivalently, $y(t)\in P(t)$ for each $t\in\T$). It means that problem
$\eqref{eq:nabla:0}$ has at least one viable solution. Since for a fixed
$v$ every solution to \eqref{eq:nabla:0} is also a solution to \eqref{eq:nabla:1}
(see, \textrm{e.g.}, \cite{Bressan+Piccoli}) we see by Theorem~\ref{th:Del:Ewa:nabla}
(the nabla Filippov theorem on time scales) that any absolutely ld-continuous
function $x:[t_0,t_1]_{\T}\rightarrow \R^n$ is a trajectory of \eqref{eq:nabla:1}
if and only if it satisfies \eqref{eq:nabla:3} almost everywhere.
\end{proof}

We can easily obtain an analogous result for delta differential inclusions.

\begin{theorem}[Existence of solutions to the delta differential inclusion \eqref{eq:2}]
\label{th:viab:delta}
Let $K:\T\twoheadrightarrow\R^n$, $K(t)=\{x\in\R^n : b_i(t)< x_i< c_i(t),\, i=1, \ldots, n\}$
with $b_i,c_i:\mathbb{T} \rightarrow \mathbb{R}$, $i=1, \ldots, n$,
delta differentiable functions on $\T$ such that $b_i(t) < c_i(t)$ for each $t\in \T$.
If every point $M\in \partial Graph(K)$ is a delta-point of strict egress for the set
$Graph(K)$ with respect to system \eqref{eq:0}, then there exists
a value $\overline{x}\in K(t_0)$ such that the initial value problem
$x^\Delta(t)\in F(t,x(t))$, $x(t_0)=\overline{x}$,
has a solution $x=\overline{x}(t)$ satisfying
$(t,\overline{x}(t))\in Graph(K)$ for every $t\in[t_0,t_1]_\T$.
\end{theorem}

\begin{proof}
The proof is analogous to the proof of Theorem~\ref{th:viab:nabla}.
\end{proof}


\section{Example}
\label{sec:ex}

Let us consider the control system
\begin{equation}
\label{ex:1}
\begin{cases}
y_1^\nabla=g_1(t, y_1,y_2,v_1,v_2):=2t^5y_1+\cos(ty_2)+y_2^8+v_1,\\
y_2^\nabla=g_2(t, y_1,y_2,v_1,v_2):=2t^6y_2+\cos(ty_1)+y_1^7+v_2,
\end{cases}
\end{equation}
defined for each $t\in[t_0,t_1]_\T$, where $t_0,t_1\in\R$, $1 \le t_0 <t_1$.
Assuming that the controls satisfy the constraint
\begin{equation}
\label{eq:constr}
v_1^2+v_2^2\leq 1,
\end{equation}
the control system \eqref{ex:1} is equivalent to the differential inclusion
\begin{equation}
\label{ex:3}
(y_1^\nabla,y_2^\nabla)\in F(y_1,y_2)
= \{(2t^5y_1+\cos(ty_2)+y_2^8+v_1,2t^6y_2
+\cos(ty_1)+y_1^7+v_2):v_1^2+v_2^2\leq1\}.
\end{equation}
Having in mind \eqref{eq:constr}, let us fix $v=(v_1,v_2)=(\frac 12,\frac 12)$.
Then the control system \eqref{ex:1} takes the form
\begin{equation}
\label{ex:2}
\begin{cases}
y_1^\nabla=g_{v1}(t, y_1,y_2):=2t^5y_1+\cos(ty_2)+y_2^8+\frac 12,\\
y_2^\nabla=g_{v2}(t, y_1,y_2):=2t^6y_2+\cos(ty_1)+y_1^7+\frac 12.
\end{cases}
\end{equation}
We will show that there exists an initial value $\bar{y}=(\bar{y}_1,\bar{y}_2)$,
$\bar{y}_i\in(-t^{-1},t^{-1})$ for $i=1,2$,
such that the initial value problem
\eqref{ex:2}, $y(t_0)=\bar{y}$, has a solution
$y=\bar{y}(t)=(\bar{y}_1(t),\bar{y}_2(t))$
satisfying \eqref{ex:3} almost everywhere with $|\bar{y}_i(t)|<t^{-1}$, $i=1,2$,
$t\in[t_0,t_1]_{\T}$. Let us define a set-valued map
$P(t):=\{(y_1,y_2)\in\R^2:\beta_i(t)>y_i>\gamma_i(t),i=1,2\}$,
where $\beta_i(t)= t^{-1}$ and $\gamma_i(t)=-t^{-1}$,
$i=1,2$, are nabla differentiable functions. We will verify that every point
$N\in \partial Graph(P)$ is a nabla-point of  strict egress for the set $Graph(P)$
with respect to system \eqref{ex:2}. Indeed, for $(t,\beta_1(t),y_2)\in Graph(P)^1_V$
and $(t,y_1,\beta_2(t))\in Graph(P)^2_V$, where $Graph(P)^i_V:=\{(t,y)\in Graph(P):V_i(t,y)=0\}$
and $V_i(t,y)=-y_i+t^{-1}$, $i=1,2$, we get
\begin{equation*}
g_{v1}(t,\beta_1(t),y_2)-\beta^\nabla_1(t)
= 2t^5\beta_1(t)+\cos(ty_2)+y_2^8+\frac 12-(t^{-1})^\nabla
\geq 2t^4-1-t^{-8}+\frac 12-(t^{-1})^\nabla>0,
\end{equation*}
\begin{equation*}
g_{v2}(t,y_1,\beta_2(t))-\beta^\nabla_2(t)
= 2t^6\beta_2(t)+\cos(ty_1)+y_1^7+\frac 12-(t^{-1})^\nabla
\geq 2t^5-1-t^{-7}+\frac 12-(t^{-1})^\nabla>0,
\end{equation*}
where $(t^{-1})^\nabla$ is negative because the function $t^{-1}$
is decreasing for every $t\in\T$. Thus, inequalities \eqref{ineq:2} hold.
Analogously, one can show that inequalities \eqref{ineq:1} hold for
$(t,\gamma_1(t),y_2)\in Graph(P)^1_W$ and $(t,y_1,\gamma_2(t))\in Graph(P)^2_W$,
where $Graph(P)^i_W:=\{(t,y)\in Graph(P):W_i(t,y)=0\}$
and $W_i(t,y)=y_i+t^{-1}$ for $i=1,2$. By Definition~\ref{def:nabla},
every point $N\in\partial Graph(P)$ is a nabla-point of strict egress for the set $Graph(P)$,
thus all assumptions of Theorem~\ref{th:viab:nabla} are fulfilled.
It implies that there exists an initial value $\bar{y}=(\bar{y}_1,\bar{y}_2)$,
$\bar{y}_i\in(-t^{-1},t^{-1})$ for $i=1,2$, such that the initial value problem
\eqref{ex:3}, $y(t_0)=\bar{y}$, has a solution $y=\bar{y}(t)$
satisfying $(t,y(t))\in Graph(P)$, $t\in[t_0,t_1]_{\T}$.


\section*{Acknowledgements}

Work supported by {\it FEDER} funds through
{\it COMPETE} --- Operational Programme Factors of Competitiveness
(``Programa Operacional Factores de Com\-pe\-ti\-ti\-vi\-da\-de'')
and by Portuguese funds through the
{\it Center for Research and Development
in Mathematics and Applications} (University of Aveiro)
and the Portuguese Foundation for Science and Technology
(``FCT --- Funda\c{c}\~{a}o para a Ci\^{e}ncia e a Tecnologia''),
within project PEst-C/MAT/UI4106/2011
with COMPETE number FCOMP-01-0124-FEDER-022690.
Girejko was also supported by FCT post-doc
fellowship SFRH/BPD/48439/2008 and by
Bialystok University of Technology grant S/WI/2/2011;
Torres by project PTDC/MAT/113470/2009.



\bigskip


\end{document}